\documentclass[11pt,twoside]{article}

\usepackage{graphicx}
\usepackage{graphics}
\usepackage{subcaption}
\usepackage{comment}
\usepackage{amsmath}
\usepackage{amsfonts}
\usepackage{amssymb}
\usepackage{euscript}
\usepackage{amsthm}
\usepackage{authblk}
\usepackage{tikz}
\usetikzlibrary{decorations}
\usetikzlibrary{decorations.pathmorphing}
\usetikzlibrary{decorations.pathreplacing}
\usetikzlibrary{decorations.markings}
\usetikzlibrary{decorations.text}
\usetikzlibrary{calc}
\usetikzlibrary[arrows,positioning, shapes, fit]
\usetikzlibrary{snakes}
\newtheorem{thm}{Theorem}[section]
\newtheorem{dfn}[thm]{Definition}

\newtheorem{cor}[thm]{Corollary}
\newtheorem{lem}[thm]{Lemma}

\newtheorem{obs}[thm]{Observation}

\newcommand{\Z}{\operatorname{Z}}
\newcommand{\F}{\operatorname{F}}

\newcommand{\mr}{\operatorname{mr}}

\newcommand{\fgp}{\bar{\gamma}_p}

\usepackage[total={6.5in, 9in}, top=1in, left=1in]{geometry}

\title{Failed power domination on graphs}
\author[1]{Abraham Glasser}
\author[2]{Bonnie Jacob\thanks{bcjntm@rit.edu}}
\author[1]{Emily Lederman}
\author[1]{Stanis\l{}aw Radziszowski}
\affil[1]{Department of Computer Science, Rochester Institute of Technology}
\affil[2]{Science and Mathematics Department, National Technical Institute for the Deaf, Rochester Institute of Technology}
\begin{document}
\maketitle

\begin{abstract}
Let $G$ be a simple graph with vertex set $V$ and edge set $E$, and let $S \subseteq V$.  The \emph{open neighborhood} of $v \in V$, $N(v)$, is the set of vertices adjacent to $v$; the \emph{closed neighborhood} is given by $N[v] = N(v) \cup \{v\}$.  The \emph{open neighborhood} of $S$, $N(S)$, is the union of the open neighborhoods of vertices in $S$, and the \emph{closed neighborhood} of $S$ is $N[S] = S \cup N(S)$.  The sets $ \mathcal{P}^i(S), i \geq 0$,   of vertices \emph{monitored} by $S$ at the $i^{\mbox{th}}$ step are given by $\mathcal{P}^0(S) = N[S]$ and $\mathcal{P}^{i+1}(S) = \mathcal{P}^i(S) \bigcup\left\{ w : \{ w \} = N[v] \backslash \mathcal{P}^i(S) \mbox{ for some } v \in \mathcal{P}^i(S) \right\}$.  If there exists $j$ such that $\mathcal{P}^j(S) = V$, then $S$ is called a \emph{power dominating set}, PDS, of $G$.  

We introduce and discuss the \emph{failed power domination number} of a graph $G$, $\fgp(G)$, the largest cardinality of a set that is not a PDS.  We prove that $\fgp(G)$ is NP-hard to compute, determine graphs in which every vertex is a PDS, and compare $\fgp(G)$ to similar parameters.  
\end{abstract}

\section{Introduction}
This paper studies power domination on graphs, which arose  because of applications to electric power networks \cite{baldwin1993power, haynes2002domination}.  We denote by $G=(V,E)$ a finite simple graph with vertex set $V$ and edge set $E$.  In cases where the graph in question is ambiguous, we use $V(G)$ and $E(G)$.  The \emph{open neighborhood} of a vertex $v \in V$, denoted $N_G(v)$ or $N(v)$ when the graph is understood, is the set of vertices adjacent to $v$; the \emph{closed neighborhood} of $v$, $N[v]$, is $N(v) \cup \{v\}$.  The \emph{open neighborhood} of a set $S \subseteq V$, denoted $N(S)$, is the union of open neighborhoods of vertices in $S$, and the \emph{closed neighborhood} of $S$, $N[S]$, is defined as  $S \cup N(S)$.   A vertex $v$ is \emph{dominated by} $S$ if $v \in N[S]$.  A set $S$ is a \emph{dominating set} if $N[S]=V$. The minimum cardinality of all dominating sets of $G$ is the \emph{domination number} $\gamma(G)$.   

Power domination differs from domination in that it contains a second step known as the \emph{propagation step}.  We use notation similar to that formalized in \cite{aazami2010domination}.  Let $i \in \mathbb{N}_0 = \left\{ 0, 1, 2, 3, \ldots \right\}$.  If $G$ is a graph and $S \subseteq V$, then the set of vertices \emph{monitored} by $S$ at Step $i$, denoted $ \mathcal{P}^i(S)$, is defined as follows.

\begin{itemize}
\item $\mathcal{P}^0(S) = N[S]$,
\item $\mathcal{P}^{i+1}(S) = \mathcal{P}^i(S) \bigcup\left\{ w : \{ w \} = N[v] \backslash \mathcal{P}^i(S) \mbox{ for some } v \in \mathcal{P}^i(S) \right\}$.
\end{itemize} 

That is, Step 0 consists of finding the set of vertices dominated by $S$. For Step $i>0$, if a vertex in $\mathcal{P}^i(S)$ has exactly one neighbor $v$ outside of $\mathcal{P}^i(S)$, then we add $v$ to $\mathcal{P}^{i+1}(S)$.  The step corresponding to $i=0$ is known as the \emph{domination step} and those corresponding to $i >0$ as the \emph{propagation steps}.  Note that for any $i \geq 0$, $\mathcal{P}^i(S) \subseteq \mathcal{P}^{i+1}(S)$.  Also, if $\mathcal{P}^{i_0+1}(S) = \mathcal{P}^{i_0}(S)$ for some $i_0$, then $\mathcal{P}^{j}(S) = \mathcal{P}^{i_0}(S)$ for any $j \geq i_0$, and then we write $\mathcal{P}^{\infty}(S) = \mathcal{P}^{i_0}(S)$.

\begin{dfn}\begin{enumerate}
\renewcommand{\theenumi}{\alph{enumi}}
\item A \emph{power dominating set} (PDS) of $G$  is a set $S \subseteq V$ such that $\mathcal{P}^{\infty}(S)=V$. 
\item A \emph{failed power dominating set} (FPDS) is a set $S \subseteq V$ such that $S$ is not a PDS.  
\item A \emph{stalled power dominating set} (SPDS) is a set $S \subseteq V$ such that $\mathcal{P}^{\infty}(S) = \mathcal{P}^0(S)$.  That is, after the domination step, no propagation steps occur.  
\item The \emph{power domination number} of $G$, denoted by $\gamma_p(G)$, is the minimum cardinality among all power dominating sets of $G$.  \item The \emph{failed power domination number} of $G$, denoted by $\bar{\gamma}_p(G)$, is the maximum cardinality among all failed power dominating sets of $G$.  
\end{enumerate}
\end{dfn}

If $S$ is an SPDS in $G$ such that $S \cup \{ u\}$ is a PDS for any vertex $u \in V \backslash S$, then we say that $S$ is \emph{maximally stalled}.  To indicate that $S$ is an SPDS and $\mathcal{P}^0(S) \subsetneq V$, we say that $S$ is \emph{properly stalled}.

\begin{figure}[h!]
\begin{center}
\begin{minipage}{3in}
\begin{center}
\begin{tikzpicture}[auto]
\tikzstyle{vertex}=[draw, circle, inner sep=0.6mm]
\node (v1) at (0,0) [vertex, label=above left:$05$] {};
\node (v2) at (0,-1) [vertex, label=above left:$04$, fill=blue] {};
\node (v3) at  (0,-2) [vertex, label=above left:$03$] {};
\node (v4) at (0,-3) [vertex, label=above left:$02$] {};
\node (v5) at (0, -4) [vertex, label=above left:$01$, fill=blue] {};
\node (v6) at (0, -5) [vertex, label=above left:$00$] {};
\node (v7) at (1,0) [vertex, label=above left:$15$] {};
\node (v8) at (1,-1) [vertex, label=above left:$14$] {};
\node (v9) at  (1,-2) [vertex, label=above left:$13$] {};
\node (v10) at (1,-3) [vertex, label=above left:$12$] {};
\node (v11) at (1, -4) [vertex, label=above left:$11$] {};
\node (v12) at (1, -5) [vertex, label=above left:$10$] {};
\node (v13) at (2,0) [vertex, label=above left:$25$] {};
\node (v14) at (2,-1) [vertex, label=above left:$24$] {};
\node (v15) at  (2,-2) [vertex, label=above left:$23$] {};
\node (v16) at (2,-3) [vertex, label=above left:$22$] {};
\node (v17) at (2, -4) [vertex, label=above left:$21$] {};
\node (v18) at (2, -5) [vertex, label=above left:$20$] {};
\node (v19) at (3,0) [vertex, label=above left:$35$] {};
\node (v20) at (3,-1) [vertex, label=above left:$34$] {};
\node (v21) at  (3,-2) [vertex, label=above left:$33$] {};
\node (v22) at (3,-3) [vertex, label=above left:$32$] {};
\node (v23) at (3, -4) [vertex, label=above left:$31$] {};
\node (v24) at (3, -5) [vertex, label=above left:$30$] {};
\node (v25) at (4,0) [vertex, label=above left:$45$] {};
\node (v26) at (4,-1) [vertex, label=above left:$44$] {};
\node (v27) at  (4,-2) [vertex, label=above left:$43$] {};
\node (v28) at (4,-3) [vertex, label=above left:$42$] {};
\node (v29) at (4, -4) [vertex, label=above left:$41$] {};
\node (v30) at (4, -5) [vertex, label=above left:$40$] {};
\node (v31) at (5,0) [vertex, label=above left:$55$] {};
\node (v32) at (5,-1) [vertex, label=above left:$54$] {};
\node (v33) at  (5,-2) [vertex, label=above left:$53$] {};
\node (v34) at (5,-3) [vertex, label=above left:$52$] {};
\node (v35) at (5, -4) [vertex, label=above left:$51$] {};
\node (v36) at (5, -5) [vertex, label=above left:$50$] {};
\foreach[evaluate=\y using int(\x-1)] \x in {2, 3, 4, 5, 6, 8, 9, 10, 11, 12, 14, 15, 16, 17, 18, 20, 21, 22, 23, 24, 26, 27, 28, 29, 30, 32, 33, 34, 35, 36}
\draw (v\y) to (v\x);
\foreach[evaluate=\y using int(\x+6), evaluate=\z using int(\y+6), evaluate=\w using int(\z+6), evaluate=\p using int(\w+6)] \x in {1, 2, 3, 4, 5, 6}
\draw  (v\p) to (v\w) to (v\z) to (v\y) to (v\x);
\foreach[evaluate=\y using int(\x+6)] \x in {25, 26, 27, 28, 29, 30}
\draw (v\y) to (v\x);
\end{tikzpicture}
\caption{A PDS $S$ in blue }
\label{fig:pds}
\end{center}
\end{minipage}
\begin{minipage}{3in}
\begin{center}
\begin{tikzpicture}[auto]
\tikzstyle{vertex}=[draw, circle, inner sep=0.6mm]
\node (v1) at (0,0) [vertex, label=above left:$05$] {};
\node (v2) at (0,-1) [vertex, label=above left:$04$] {};
\node (v3) at  (0,-2) [vertex, label=above left:$03$, fill=blue] {};
\node (v4) at (0,-3) [vertex, label=above left:$02$, fill=blue] {};
\node (v5) at (0, -4) [vertex, label=above left:$01$,  fill=blue] {};
\node (v6) at (0, -5) [vertex, label=above left:$00$, fill=blue] {};
\node (v7) at (1,0) [vertex, label=above left:$15$] {};
\node (v8) at (1,-1) [vertex, label=above left:$14$] {};
\node (v9) at  (1,-2) [vertex, label=above left:$13$] {};
\node (v10) at (1,-3) [vertex, label=above left:$12$, fill=blue] {};
\node (v11) at (1, -4) [vertex, label=above left:$11$, fill=blue] {};
\node (v12) at (1, -5) [vertex, label=above left:$10$, fill=blue] {};

\node (v13) at (2,0) [vertex, label=above left:$25$, fill=blue] {};
\node (v14) at (2,-1) [vertex, label=above left:$24$] {};
\node (v15) at  (2,-2) [vertex, label=above left:$23$] {};
\node (v16) at (2,-3) [vertex, label=above left:$22$] {};
\node (v17) at (2, -4) [vertex, label=above left:$21$, fill=blue] {};
\node (v18) at (2, -5) [vertex, label=above left:$20$, fill=blue] {};
\node (v19) at (3,0) [vertex, label=above left:$35$, fill=blue] {};
\node (v20) at (3,-1) [vertex, label=above left:$34$, fill=blue] {};
\node (v21) at  (3,-2) [vertex, label=above left:$33$] {};
\node (v22) at (3,-3) [vertex, label=above left:$32$] {};
\node (v23) at (3, -4) [vertex, label=above left:$31$] {};
\node (v24) at (3, -5) [vertex, label=above left:$30$, fill=blue] {};

\node (v25) at (4,0) [vertex, label=above left:$45$, fill=blue] {};
\node (v26) at (4,-1) [vertex, label=above left:$44$, fill=blue] {};
\node (v27) at  (4,-2) [vertex, label=above left:$43$, fill=blue] {};
\node (v28) at (4,-3) [vertex, label=above left:$42$] {};
\node (v29) at (4, -4) [vertex, label=above left:$41$] {};
\node (v30) at (4, -5) [vertex, label=above left:$40$] {};
\node (v31) at (5,0) [vertex, label=above left:$55$, fill=blue] {};
\node (v32) at (5,-1) [vertex, label=above left:$54$, fill=blue] {};
\node (v33) at  (5,-2) [vertex, label=above left:$53$, fill=blue] {};
\node (v34) at (5,-3) [vertex, label=above left:$52$, fill=blue] {};
\node (v35) at (5, -4) [vertex, label=above left:$51$] {};
\node (v36) at (5, -5) [vertex, label=above left:$50$] {};

\foreach[evaluate=\y using int(\x-1)] \x in {2, 3, 4, 5, 6, 8, 9, 10, 11, 12, 14, 15, 16, 17, 18, 20, 21, 22, 23, 24, 26, 27, 28, 29, 30, 32, 33, 34, 35, 36}
\draw (v\y) to (v\x);

\foreach[evaluate=\y using int(\x+6), evaluate=\z using int(\y+6), evaluate=\w using int(\z+6), evaluate=\p using int(\w+6)] \x in {1, 2, 3, 4, 5, 6}
\draw  (v\p) to (v\w) to (v\z) to (v\y) to (v\x);

\foreach[evaluate=\y using int(\x+6)] \x in {25, 26, 27, 28, 29, 30}
\draw (v\y) to (v\x);

\end{tikzpicture}
\caption{An FPDS and SPDS $S$ in blue }
\label{fig:fpds}
\end{center}
\end{minipage}
\end{center}
\end{figure}

In Figure \ref{fig:pds}, the set $S= \{ 04, 01 \}$ is a PDS, while in Figure \ref{fig:fpds}, the blue vertices represent an FPDS and an SPDS (since after the dominating step, all vertices will be monitored except the main diagonal).  Thus, $\gamma_p(G) \leq 2$, but $\fgp(G) \geq 20$.  

Given a graph $G$ with $S \subseteq V$, we use $G[S]$ to denote the graph induced by the set $S$.  Given graphs $G$ and $H$, the \emph{join of $G$ and $H$}, denoted $G \vee H$,  has the vertex set $V(G) \cup V(H)$ and edge set $E(G) \cup E(H) \cup \{ \{u,v\} : u \in V(G) \mbox{ and } v \in V(H)\}$.  The \emph{vertex connectivity} $\kappa(G)$ of graph $G$ is the minimum number of vertices whose removal causes $G$ to be disconnected.  

In this paper, we determine the computational complexity of testing $\fgp(G) \geq k$ and find graphs with extreme values of $\fgp(G)$.  We present a list of graphs that have $\fgp(G)=0$, which is a particularly interesting case, since $\fgp(G)=0$ implies that any nonempty set of vertices in $G$ is a PDS.  We also discuss the relationship between $\fgp(G)$ and some related parameters in the literature.

\section{Motivation and related parameters}

The idea of failed power domination on graphs is motivated by the need to monitor electric power networks.  In \cite{baldwin1993power}, the authors describe the problem of observing a power system while minimizing the number of measurement devices known as phasor measurement units (PMUs) on the network. A PMU measures the voltage and phase angle, and allows for synchronization \cite{nuqui2005phasor}, which is one strategy described in \cite{heydt2001solution} for making the power grid more robust.  If a PMU measures the voltage and phase angle of vertex $v$ (or edge $e$), then $v$ (or $e$) is said to be \emph{observed}.  The vertex on which a PMU is placed is observed, as are its incident edges and adjacent vertices.  In addition, any vertex that is incident to an observed edge is observed; any edge joining two observed vertices is observed; finally, from Kirchhoff's Law, given an observed vertex $v$ with $k$ incident edges, if $k-1$ of the edges are observed, then all $k$ are observed.  In \cite{haynes2002domination} the authors formulate and investigate this problem as a graph theoretic problem.  Later, Kneis et al. \cite{kneis2006parameterized} showed that the problem can be simplified to omit any reference to edges.  The formal set definition of $\mathcal{P}^i(S)$ was introduced in \cite{aazami2010domination}.  

Under the model described in \cite{haynes2002domination}, the power domination number $\gamma_p(G)$ gives the minimum number of PMUs required to observe a power network represented by graph $G$.  The power domination number has been studied for multiple families of graphs \cite{benson2018zero, ferrero2017note, varghese2018power}, as has the complexity of $\gamma_p(G)$ \cite{haynes2002domination}.   On the other hand, the failed power domination number $\fgp(G)$ that we defined above gives a worst case scenario: what is the maximum number of PMUs that we could use on a given network represented by $G$, but fail to observe the full network?  In addition, $\fgp(G)+1$ gives us the minimum number of PMUs necessary to successfully observe the full network no matter where we place the PMUs.  

The concept of zero forcing, while related to power domination, was introduced in 2008  in the context of minimum rank problems \cite{aim2008zero,barioli2010zero} as well as quantum networks in 2007  \cite{burgarth2009local, burgarth2013zero, burgarth2007full, burgarth2009indirect, severini2008nondiscriminatory}.
Zero forcing acts like power domination, but without the domination step.  
That is, given a set $S$, $\mathcal{Q}^0(S) = S$, and for $i \in \mathbb{N}_0$, $\mathcal{Q}^{i+1}(S)= \mathcal{Q}^i(S) \bigcup \left\{ N[v] : v \in \mathcal{Q}^i(S) \mbox{ and } \left| N[v] \backslash \mathcal{Q}^i(S) \right| = 1 \right\}$.  Note that there exists an $i_0$ such that for all $j>i_0$, $\mathcal{Q}^j(S) = \mathcal{Q}^{i_0}(S)$, so we write $\mathcal{Q}^{\infty}(S)=\mathcal{Q}^{i_0}(S)$.  If $\mathcal{Q}^{\infty}(S) = V$, then $S$ is a \emph{zero forcing set}.  Otherwise, $S$ is a \emph{failed zero forcing set}. %citations
The smallest cardinality of any zero forcing set in $G$ is the \emph{zero forcing number} $\Z(G)$, and the largest cardinality of any failed zero forcing set is the \emph{failed zero forcing number} $\F(G)$ \cite{ansill2016failed, fetcie2015failed}.  
Complexity results for failed zero forcing were established in \cite{shitov2017complexity}.

\begin{lem}
For a graph $G=(V,E)$, suppose $S \subseteq S' \subseteq V$.  Then, 
\begin{enumerate}
\item $\mathcal{Q}^{\infty}(S) \subseteq \mathcal{Q}^{\infty}(S')$, \label{monotonicityforq} and
\item $\mathcal{P}^{\infty}(S) \subseteq \mathcal{P}^{\infty}(S')$. \label{monotonicityforp}
\end{enumerate}  
\end{lem}

\begin{proof}
Suppose $S \subseteq S' \subseteq V$.  Then $\mathcal{Q}^0(S) \subseteq  \mathcal{Q}^0(S')$.  Assume $\mathcal{Q}^k(S)  \subseteq \mathcal{Q}^k(S')$.  If $u \in \mathcal{Q}^{k+1}(S)$, then either $u \in \mathcal{Q}^k(S)$, implying $u \in \mathcal{Q}^{k+1}(S')$, or there exists $v \in \mathcal{Q}^k(S)$ such that $N[v] \backslash \mathcal{Q}^k(S) =\{ u\}$ giving us $N[v] \backslash \{u\} \subseteq \mathcal{Q}^k(S')$.  Thus $u \in \mathcal{Q}^{k+1}(S')$.  Hence $\mathcal{Q}^{i}(S) \subseteq \mathcal{Q}^{i}(S')$ for any $i \in \mathbb{N}_0$, giving us  $\mathcal{Q}^{\infty}(S) \subseteq \mathcal{Q}^{\infty}(S')$.  

To prove \ref{monotonicityforp},   let $u \in \mathcal{P}^0(S)$.  Then $u \in N[v]$ for some $v \in S$.  Since $S \subseteq S'$, we have $u \in \mathcal{P}^0(S')$.  Thus, $\mathcal{P}^0(S) \subseteq \mathcal{P}^0(S')$.  The remainder of the proof is identical to the proof  of \ref{monotonicityforq}.  
\end{proof}

Since any set $S \subseteq V$ is a subset of the set of vertices it dominates, we have the following observation.

\begin{obs}
$\fgp(G) \leq \F(G)$.  
\end{obs}

\section{Complexity}

In this section, we show that it is NP-hard to determine whether  $G$ has a failed power dominating set of cardinality at least $k$.  We use a similar technique to the one used in \cite{shitov2017complexity} to show NP-completeness of failed zero forcing parameters.  

\vskip0.2in 
FAILED POWER DOMINATING SET (FPDS), $(G,m)$

Instance: Graph $G=(V,E)$ and a positive integer $m$ %will we have problems with $k=1$?

Question: Does $G$ have a proper stalled subset of cardinality at least $m$?

\vskip0.2in

To prove that FPDS is NP-hard, we construct a polynomial reduction from the well-known NP-complete problem, INDEPENDENT SET, which remains NP-complete when restricted to connected graphs \cite{shitov2017complexity}.%cite, and carp, reducibility among... in ref from shiv

\vskip0.2in

INDEPENDENT SET, $(G,k)$

Instance: Connected graph $G=(V,E)$  and a non-negative integer $k$. 

Question: Does $G$ contain an independent set of cardinality $k$?

\vskip0.2in

The domination number of a path on $k$ vertices, $\gamma(P_k)$,  is known to be $ \lceil k/3 \rceil$ \cite{frendrup2009upper}.

\begin{lem}
Let $G$ be a graph that contains an induced subgraph $P_k$, where $k \geq 3$, all internal vertices of $P_k$ have degree $2$ in $G$, one end vertex of $P_k$ has degree $1$ in $G$, and the other end vertex, $v$, has degree at least $3$ in $G$.  If $S$ is an SPDS containing at least one vertex of $P_k$, then $|S \cap P_k| \geq \gamma(P_k) = \lceil k/3 \rceil $.  
If $S$ is maximally stalled and contains at least one vertex of $P_k$, then $|S \cap P_k| \geq k-1$.  \label{lem:pathspread}\end{lem}

\begin{proof}
Note that if there are at least two adjacent vertices in $\mathcal{P}^0(S) \cap P_k$, then for some $i \geq 0$, $V(P_k) \subseteq \mathcal{P}^i(S)$.  If there is a vertex in $S \cap P_k$, then after the domination step, there are at least two adjacent vertices from $P_k$ in $\mathcal{P}^0(S)$.  Thus, if $S$ is stalled, it must be that at least $\gamma(P_k)$ vertices on the path are in $S$; otherwise, $\mathcal{P}^1(S) \backslash \mathcal{P}^0(S)$ is nonempty.      

Since at least $\gamma(P_k)$ vertices on the path $P_k$ are in $S$, it follows that $\mathcal{P}^0(S)$ contains all vertices in $P_k$.  Thus, if $S$ is maximally stalled and contains at least one vertex of $P_k$, it must contain all vertices other than $v$.  That is, $|S \cap P_k| \geq k-1$.
\end{proof}

To prove the following lemma,  we construct a polynomial reduction from INDEPENDENT SET.  An example reduction instance is shown in Figures \ref{fig:independent} and \ref{fig:fpdsgprime}.

\begin{figure}[h!]
\begin{center}
\begin{minipage}{2in}
\begin{center}
\begin{tikzpicture}[auto,scale=1.2]
\tikzstyle{vertex}=[draw, circle, inner sep=0.8mm]
\node (v1) at (0,0) [vertex, fill=blue] {};
\node (v2) at (1,0) [vertex] {};
\node (v3) at (2,0.5) [vertex, fill=blue] {};
\node (v4) at (2,-0.5) [vertex] {};
\draw (v1) --   (v2) node [midway] {$e$};
\draw (v3) -- (v2) node [above, midway] {$f$};
\draw (v4) -- (v2) node [midway] {$g$};
\draw (v3) -- (v4) node [midway] {$h$};;
\end{tikzpicture}
\caption{A graph $G$ with independent set of cardinality $2$ shown in blue}
\label{fig:independent}
\end{center}
\end{minipage}
\begin{minipage}{4in}
\begin{center}
\begin{tikzpicture}[auto, scale=2.6]
\tikzstyle{vertex}=[draw, circle, inner sep=0.8mm]
\node (v1) at (0,0) [vertex, fill=blue] {};
\node (v2) at (1,0) [vertex] {};
\node (v3) at (2,0.5) [vertex, fill=blue] {};
\node (v4) at (2,-0.5) [vertex] {};
\node (u0) at (0.5,0) [vertex, label=below:$v_{e_0}$] {};
\draw(v1)--(u0)--(v2);
\foreach \x in {1, 2}
\node (u\x) at (0.75-0.25*\x, 0.25+0.25*\x) [vertex, label=$v_{e_{\x}}$, fill=blue]{};
\draw (u0) -- (u1) node [midway, right] {$e_{1}$};
\foreach \x in {15, 16}
\node (u\x) at (3.75-0.25*\x, -2.75+0.25*\x) [vertex, label=$v_{e_{\x}}$, fill=blue]{};
\foreach[evaluate=\y using int(\x-1)]\x in {2, 16}
\draw (u\x) -- (u\y) node [midway, below left] {$e_{\x}$};
\draw [dotted]  (u2) to (u15);
\node (w0) at (1.5,0.25) [vertex, label=left:$v_{f_0}$] {};
\draw(v3)--(w0)--(v2);
\foreach \x in {1, 2}
\node (w\x) at (1.5-0.25*\x, 0.25+0.25*\x) [vertex, label=left:$v_{f_{\x}}$, fill=blue]{};
\foreach \x in {15,16}
\node (w\x) at (4.5-0.25*\x, -2.75+0.25*\x) [vertex, label=left:$v_{f_{\x}}$, fill=blue]{};
\foreach[evaluate=\y using int(\x-1)]\x in {1, 2, 16}
\draw (w\x) -- (w\y) node [midway] {$f_{\x}$};
\draw [dotted] (w2) to (w15);
\node (x0) at (1.5,-0.25) [vertex, label=right:$v_{g_{0}}$] {};
\draw(v4)--(x0)--(v2);
\foreach \x in {1, ...,2}
\node (x\x) at (1.5-0.25*\x, -0.25-0.25*\x) [vertex, label=$v_{g_{\x}}$, fill=blue]{};
\foreach \x in {15, 16}
\node (x\x) at (4.5-0.25*\x, 2.75-0.25*\x) [vertex, label=$v_{g_{\x}}$, fill=blue]{};
\foreach[evaluate=\y using int(\x-1)]\x in {1,2, 16}
\draw (x\x) -- (x\y) node [midway, below right] {$g_{\x}$};
\draw [dotted] (x2) to (x15);
\node (y0) at (2,0) [vertex, label=above left:$v_{h_0}$] {};
\draw(v4)--(y0)--(v3);
\foreach \x in {1, 2}
\node (y\x) at (2.5, 0.3-0.3*\x) [vertex, label=right:$v_{h_{\x}}$, fill=blue]{};
\draw (y0) -- (y1) node [midway, above] {$h_{1}$};
\foreach \x in {15, ...,16}
\node (y\x) at (2.5, 3.7-0.3*\x) [vertex, label=right:$v_{h_{\x}}$, fill=blue]{};
\foreach[evaluate=\y using int(\x-1)]\x in {2,  16}
\draw (y\x) -- (y\y) node [midway, left] {$h_{\x}$};
\draw [dotted] (y2) to (y15);
\node (x) at (1.5,0) [vertex, label=left:$x$]{};
\draw(x0)--(x);
\draw(y0)--(x);
\draw(w0)--(x);
\draw(u0) to [out=30, in=150] (x);
\end{tikzpicture}
\caption{The graph $G'$ with FPDS $S$ in blue, $|S|=66$}
\label{fig:fpdsgprime}
\end{center}
\end{minipage}
\end{center}
\end{figure}

\begin{lem}
FAILED POWER DOMINATING SET is NP-hard.  \label{lem:nphard}
\end{lem}

%uses Shitov (On the complexity of failed zero forcing) construction
%copied from old version of paper, check with nice version of shitov

\begin{proof}

Suppose $(G,k)$ with $n=|V|  \geq 3$,  $k \geq 2$ is an instance of INDEPENDENT SET.  We construct from it an instance $(G',m)$  of FPDS for $m= n^2 |E|+k$.  Let $U$ be an independent set of $G$.  Then $U'= U \cup V'_1 \cup V'_2 \cup \cdots \cup V'_{n^2}$ is an SPDS of cardinality $n^2|E| + k$ in $G'$, where $G'=(V',E')$ is constructed as follows.  

\begin{enumerate}
\item $V \subseteq V'$.
\item Subdivide every edge of $G$.  That is, for each $e = \{ u, v\} \in E$, add a vertex $v_{e_0}$ to $V'$, and let $\{u, v_{e_0}\}, \{v_{e_0}, v\} \in E'$. Let $V'_0$ denote these added vertices, and $E'_0$ the added edges.
\item For each  $e = \{ u, v\} \in E$, add vertices $v_{e_1}$ through $v_{e_{n^2}}$ to $V'$.  For each $i=1, 2, \ldots, n^2$, let $V'_i$ denote all $v_i \in V'$, and add edge $e_i = \{v_{e_{i-1}}, v_{e_{i}}\}$ to $E'$.  Let $P(e)$ denote the path from $v_{e_0}$ to $v_{e_{n^2}}$. Let the set of all such paths be denoted $\rho$.  
\item Add a vertex $x$ to $V'$.  For each vertex $v_{e_0} \in V'_0$, add $\{x, v_{e_0}\}$ to $E'$.  
\end{enumerate}

To see that $U'$ is an SPDS in $G'$, note that  $\mathcal{P}^0(U')= U \cup V'_0 \cup V'_1 \cup V'_2 \cup \cdots \cup V'_{n^2}$.  
%so same argument as in shiv
If $U'$ is not an SPDS, then $\mathcal{P}^1(U') \backslash \mathcal{P}^0(U')$ is nonempty.  
The only vertices in $V' \backslash   \mathcal{P}^0(U')$ are $x$ and the vertices from $V \backslash U$.  We know that $N_{G'}(x)=V_0'$, but each vertex in $V'_0$ has at least one other neighbor in $V \backslash (U  \cup  \mathcal{P}^0(S))$ (since $U$ is an independent set in $G$).  Hence, $x \notin \mathcal{P}^1(U')$.  Similarly, for any vertex $v \in V \backslash U$, the neighborhood $N_{G'}(v)$ is contained in $V_0'$.  But for each $v_{e_0} \in V'_0$,  $N_{G'}\left(v_{e_0}\right)$ includes $x$ and one vertex from $V \backslash U$.   Hence if $v \in V \backslash U$, then $v \notin \mathcal{P}^1(U')$, and $U'$ is stalled.  

Suppose that $S\subseteq V(G')$ is maximally stalled with $|S| \geq n^2|E| + 2$.  We will show that for each path $P(e) \in \rho$, $|S \cap P(e)| \geq n^2$.  Since $|V'| = (n^2 + 1)|E| + n + 1$, there are at most $|E| + n -1$ vertices in $V' \backslash S$.  Each path $P(e)$ has $n^2 + 1$ vertices.  Note that 
$n^2 + 1 - (|E| + n -1)  \geq n^2 - n + 2 - \frac{n(n-1)}{2} >1$, and
thus, $P(e)$ contains at least one vertex in $S$.  By Lemma \ref{lem:pathspread}, then, $|S \cap P(e)| \geq n^2$, implying that $\cup_{i=0}^{n^2} V'_i \subseteq \mathcal{P}^0(S)$.  

We show that $V_0' \cap S = \emptyset$.  Without loss of generality suppose $v_{e_0} \in S$.   Then $N_{G'}(v_{e_0}) = \{u, v, x, v_{e_1}\}$ where $e=\{u,v\}$, so $\{ u, v, x\} \subseteq \mathcal{P}^0(S)$.  Since $G$ is connected, $G'$ is also connected.  Thus, there is a path in $G'$ from $u$  to any vertex in $V$.  Since $S$ is properly stalled and $(V' \backslash V)\subseteq \mathcal{P}^0(S)$, it follows that there must be some vertex $y \in V$ such that $y \notin \mathcal{P}^0(S)$.  Then, on the path from $u$ to $y$, there exists  $\hat{e}=\{w,z\} \in E$ with $w \in \mathcal{P}^0(S)$ and $z \notin \mathcal{P}^0(S)$.   Consider the vertex $v_{\hat{e}_0} \in V'_0$.  The set $N_{G'}[v_{\hat{e}_0}] \backslash \mathcal{P}^0(S)$ consists only of the vertex $z$ (since we just noted that $x \in \mathcal{P}^0(S)$),  so $z \in \mathcal{P}^1(S) \backslash \mathcal{P}^0(S)$, a contradiction of $S$ being stalled.  Hence, $V'_0 \cap S = \emptyset$.  
Since we know that for each path $P(e) \in \rho$, $|P(e) \cap S| \geq n^2$, it follows that $\cup_{i=1}^{n^2} V'_i \subseteq S$.  

Now, we show that $x \notin S$.  Suppose $x \in S$.  Note that $V \cap S$ is nonempty, because we assumed that $|S| \geq n^2|E|+2$.  Also, $V \backslash S$ is nonempty since we showed that $\cup_{i=0}^{n^2} V'_i \subseteq \mathcal{P}^0(S)$.  Since we're assuming that $x \in S$, if $V \subseteq S$, then $\mathcal{P}^0(S) = V'$, contradicting the assumption that $S$ is properly stalled.  Hence, there exists an edge $e=\{u,v\} \in E$ with $u \in S$ and $v \notin S$.    Then the vertex $v_{e_0}$ has $N_{G'}[v_{e_0}] \backslash \mathcal{P}^0(S) = \{v\}$, and $\mathcal{P}^1(S) \backslash \mathcal{P}^0(S)$ is nonempty, a contradiction of $S$ being stalled.  Hence, $x \notin S$.  

Finally, we will show that $S \cap V$ is an independent set of $G$.  Suppose there exists an edge $e=\{u, v \} \in E$ for some $u, v \in S$.  Then $v_{e_0}$ has $N_{G'}[v_{e_0}] \backslash \mathcal{P}^0(S) = \{x\}$, and $x \in \mathcal{P}^1(S) \backslash \mathcal{P}^0(S)$, a contradiction of $S$ being stalled.  Hence, $S\cap V$ is an independent set in $G$.  

This gives us that for any maximal properly stalled subset $S$ of $V'$, 
$$|S| = n^2 |E| + t,$$ 
where $t$ is the order of independent set $S \cap V$.  Thus $G'$ has an SPDS of order $m=n^2|E|+k$ if and only if $G$ has an independent set of order $k$. The construction of $G'$ is polynomial and thus this completes our proof that FPDS is NP-hard.
\end{proof}

For a graph $G$, positive integer $k$, and $S\subseteq V$ with $|S| \leq k$, it is verifiable in polynomial time whether or not $S$ is a PDS  \cite{haynes2002domination}.  Thus it is verifiable in polynomial time whether $S$ is an FPDS, completing the proof of the following theorem.  

\begin{thm}
FAILED POWER DOMINATING SET is NP-complete.  
\end{thm}

\section{Extreme values}

In this section, we characterize $n$-vertex graphs $G$ with $\fgp(G)\geq n-3$.  We also give some results for the case  $\fgp(G)=0$.  The next observation follows from the definition of PDS.

\begin{obs}
If $S$ is a PDS of $G$, then $\mathcal{P}^0(S) \backslash S$ is a zero forcing set of $G[V \backslash S]$.  \label{obs:domset}
\end{obs}
%This is a result of the fact that any vertex that zero forces another may have any number of filled neighbors, and also that the domination step will fill all neighbors of the original set.  

\begin{thm} We have the following characterization of graphs with high values of $\fgp(G)$.   \label{thm:highvalues}
\begin{enumerate}
\item $\fgp(G) = n-1$ if and only if $G$ has an isolated vertex. \label{nminus1}
\item $\fgp(G) = n-2$ if and only if $G$ contains $K_2$ as a component, and no isolated vertices. \label{nminus2}
\item $\fgp(G) = n-3$ if and only if $G$ contains no components that are isolated vertices or $K_2$ and $G=(V,E)$ contains as an induced subgraph 
\begin{itemize}
\item  $P_3$ where only the middle vertex may be adjacent to other vertices in $V,$ or
\item $K_3$ where at most one of the vertices may be adjacent to other vertices in $V.$
\end{itemize} \label{nminus3}
\end{enumerate}
\end{thm}
\begin{proof}
If $G$ has an isolated vertex $v$, let $S= V \backslash \{v\}$.  Then $S$ is an FPDS, and $\fgp(G) = n-1$.  Conversely, let $\fgp(G) = n-1$, and let $S$ be an FPDS.  If the single vertex $v \in V \backslash S$ has an edge to any vertex $u \in S$, then $v \in \mathcal{P}^0(S)$.  Hence, $v$ is isolated, completing the proof of \ref{nminus1}.

If $G$ contains no isolated vertices, and one component is $K_2$ with vertices $u,v$, then let $S = V\backslash \{u,v\}$. Then $S$ is an FPDS, and $\fgp(G)=n-2$.  

Conversely, suppose $\fgp(G) = n-2$.  We know $G$ contains no isolated vertices.  Let $S$ be an FPDS with $|S| = n-2$.  Let $u,v$ be the two vertices in $V\backslash S$.  If $u$ is adjacent to some vertex $w \in S$,   then  $u \in \mathcal{P}^0(S)$, giving us that all vertices except possibly $v$ are in $\mathcal{P}^0(S)$.  But then, $v \in \mathcal{P}^1(S)$, implying that $S$ is a PDS.  Therefore, neither $u$ nor $v$ is adjacent to any vertex in $S$, but since there are no isolated vertices, $uv$ forms a copy of $K_2$, completing the proof of \ref{nminus2}.

If $G$ does not contain any isolated vertex or  component that is $K_2$, then $\fgp(G) \leq n-3$.  If $G$ contains an induced copy of $P_3=\{u, v, w\}$ with edges $uv, vw$, note that  only $v$ may be adjacent to other vertices in $V$. Let $S = V \backslash \{u, v, w\}$.  Then it is possible that $v \in \mathcal{P}^0(S)$, but $u, w \notin \mathcal{P}^i(S)$ for any $i \geq 0$ since $N(u)=N(w)=\{v\}$.  The same holds if $G[\{u,v,w\}]$ forms a copy of $K_3$.  

Conversely, suppose $\fgp(G)=n-3$, and let $S$ be an FPDS with $|S|=n-3$.  Let $\{u,v,w\} = V\backslash S$.  Suppose $\{ u, v \} \subseteq N(S)$.    Then $w \notin N(S)$, because $w \in N(S)$ implies $\{u,v,w\} \subseteq \mathcal{P}^0(S)$.  However, since $w$ cannot be an isolated vertex,  $w \in N(u)$ (without loss of generality) but then $w \in \mathcal{P}^1(S)$, implying that $S$ is a PDS.   Hence, only one of $\{u,v,w\}$ may be in $N(S)$.  Without loss of generality, say it is $v$.  Since  $G$ has no isolated vertices or $K_2$ component, and vertices $u$ and $w$ have no neighbors outside of $\{u,v,w\}$, then $G[\{u,v,w\}]$ is either $K_3$ or $P_3$.  If it is $K_3$, we are done.  If it is $P_3$, and $v$ has any other neighbors in $G$, note that $v$ must be the middle vertex.  If not, $\{u,w\} \subseteq \mathcal{P}^2(S)$, implying $S$ is not an FPDS.  This completes the proof of \ref{nminus3}.  
\end{proof}

%could we do 4?  can't have special vertex v adjacent to three leaves, because that would be n-3, or the endpoint of a P4, but all others?  ....maybe too complicated

\subsection{Graphs in which every  vertex is a PDS}
\label{sec:lowvalues}
%Find all graphs with $\fg$, $\fgp$ equal to 1, 2, ...

In this section, we present some results on graphs that have $\fgp(G)=0$.  Note that if $\fgp(G)=0$, then any single vertex is a PDS of $G$.  We use the notation $\mathcal{P}^i_G(S)$ to indicate $\mathcal{P}^i(S)$ in $G$ only when the graph in question is ambiguous.  

\begin{lem}
$\bar{\gamma}_p\left( G_1 \vee G_2 \vee \cdots \vee G_n \right) =0$ if and only if for each $i=1, 2, \ldots, n$, where $n \geq 2$, either $\bar{\gamma}_p(G_i)=0$ or $G_i=\overline{K_2}$. \label{lem:join}
\end{lem}
\begin{proof}
Let $G_1$ and $G_2$ be graphs, and let $v \in V(G_1)$.  Then $\mathcal{P}^0_{G_1 \vee G_2}(\{v\})= \mathcal{P}^0_{G_1}(\{v\}) \cup V(G_2)$, and as a result, $\mathcal{P}^i_{G_1 \vee G_2}(\{v\}) = \mathcal{P}^i_{G_1}(\{v\}) \cup V(G_2)$ for any $i \geq 0$, unless $G_1=\overline{K_2}$, in which case $\mathcal{P}^1_{G_1 \vee G_2}(\{v\}) = V(G_1) \cup V(G_2)$.  Hence, $\{v\}$ is a PDS in $G_1 \vee G_2$ if and only if  $\{v\}$ is a PDS in $G_1$ or $G_1=\overline{K_2}$, and similarly for $G_2$.  That is, $\bar{\gamma}_p\left( G_1 \vee G_2\right) = 0$ if and only if $\bar{\gamma}_p(G_1)=0$ or $G_1 = \overline{K_2}$, and  $\bar{\gamma}_p(G_2)=0$ or $G_2 = \overline{K_2}$.
We can use the same argument if $G_1$ or $G_2$ is itself the join of two graphs.  Hence, by induction, $\bar{\gamma}_p\left( G_1 \vee G_2 \vee \cdots \vee G_n \right) =0$ if and only if $\bar{\gamma}_p(G_i)=0$ or $G_i = \overline{K_2}$ for each $i=1, 2, \ldots, n$.
\end{proof}

In a poster \cite{tostado2017failed}, Tostado listed several families of graphs that have $\fgp=0$.  We include this list here with proofs.  For $n \geq 4$, a \emph{wheel on $n$ vertices}, $W_n$, is defined by $W_n = C_{n-1} \vee \{v\}$.

\begin{thm}
The following graphs have $\fgp=0$ \cite{tostado2017failed}\label{thm:tostado}.
\begin{enumerate}
\item a path on $n$ vertices, $P_n$ for $n \geq 1$, \label{path}
\item a cycle on $n$ vertices, $C_n$ for $n \geq 3$,  \label{cycle}
\item a complete graph on $n$ vertices, $K_n$, for $n \geq 1$, \label{complete}
\item a wheel on $n$ vertices, $W_n$ for $n\geq 4$. \label{wheel}
\end{enumerate}
\end{thm}

\begin{proof}
If $G=P_n$ or $G=C_n$, and $S=\{ v\}$ for any vertex $v \in V(G)$, then $\mathcal{P}^0(S)$ consists of at least two adjacent vertices, and $S$ is a PDS, proving \ref{path} and \ref{cycle}.  For \ref{complete}, note that $K_n = G_1 \vee G_2 \vee \cdots \vee G_n$ where $G_i$ consists of a single vertex, and thus by Lemma \ref{lem:join}, $\fgp(K_n)=0$.  Finally, since $W_n = C_{n-1} \vee \{v\}$, also by Lemma \ref{lem:join}, $\fgp(W_n)=0$.  Thus  \ref{wheel} also holds.
\end{proof}

We add several families of graphs to this list.  An example for item \ref{fanchordsplus1}   from Theorem  \ref{thm:listwith0} below is shown in Figure \ref{fig:fanchords}.  In the proof of Theorem \ref{thm:listwith0}, we use a property that follows from the definitions of PDS and zero forcing sets:

\begin{lem}
Suppose $G$ is a graph, and $S \subseteq V$.  Suppose that for some $i \geq 0$, a subset $S'$ of 
$\mathcal{P}^i(S)$ is a zero forcing set of the graph induced by $(V \backslash \mathcal{P}^i(S) ) \cup S'$.  Then $S$ is a PDS of $G$.   \label{lem:zfs}
\end{lem}

\begin{proof}
For each $j \geq 0$, $\mathcal{Q}^j(S') \subseteq  \mathcal{P}^{i+j}(S)$.  Thus, if $\mathcal{Q}^{\infty}(S') = V$, then $\mathcal{P}^{\infty}(S') = V$ as well.  
\end{proof}

\begin{figure}[h!]
\begin{center}
\begin{subfigure}{0.19\linewidth}
\begin{center}
\begin{tikzpicture}[auto, scale=1]
\tikzstyle{vertex}=[draw, circle, inner sep=0.6mm]

\node (v3) at (108:.8)[vertex, fill=blue, label= $v_3$]{};

\foreach[evaluate=\y using 36*\x)] \x in {2}
\node (v\x) at (\y:.8)[vertex, label=$v_{\x}$]{};

\foreach[evaluate=\y using 36*\x)] \x in {4, 5, 6}
\node (v\x) at (\y:.8)[vertex, label=left: $v_{\x}$]{};

\foreach[evaluate=\y using 36*\x)] \x in {7, 8}
\node (v\x) at (\y:.8)[vertex, label=below: $v_{\x}$]{};

\foreach[evaluate=\y using 36*\x)] \x in {1, 9, 10}
\node (v\x) at (\y:.8)[vertex, label=right: $v_{\x}$]{};

\foreach[evaluate=\y using int(\x+1))] \x in {1, ..., 9}
\draw (v\x) -- (v\y);

\draw(v10)--(v1);

\foreach \x in {6,7,8}
\draw (v\x) to (v1);
\draw (v2) to (v5);

\end{tikzpicture}
\end{center}
\caption{$S$}
\end{subfigure}
\begin{subfigure}{0.19\linewidth}
\begin{center}
\begin{tikzpicture}[auto, scale=1]
\tikzstyle{vertex}=[draw, circle, inner sep=0.6mm]

\node (v3) at (108:.8)[vertex, fill=blue, label= $v_3$]{};

\foreach[evaluate=\y using 36*\x)] \x in {2}
\node (v\x) at (\y:.8)[vertex, fill=blue, label=$v_{\x}$]{};

\foreach[evaluate=\y using 36*\x)] \x in {4}
\node (v\x) at (\y:.8)[vertex, fill=blue, label=left: $v_{\x}$]{};

\foreach[evaluate=\y using 36*\x)] \x in {5, 6}
\node (v\x) at (\y:.8)[vertex, label=left: $v_{\x}$]{};

\foreach[evaluate=\y using 36*\x)] \x in {7, 8}
\node (v\x) at (\y:.8)[vertex, label=below: $v_{\x}$]{};

\foreach[evaluate=\y using 36*\x)] \x in {1, 9, 10}
\node (v\x) at (\y:.8)[vertex, label=right: $v_{\x}$]{};

\foreach[evaluate=\y using int(\x+1))] \x in {1, ..., 9}
\draw (v\x) -- (v\y);

\draw(v10)--(v1);

\foreach \x in {6,7,8}
\draw (v\x) to (v1);
\draw (v2) to (v5);
\end{tikzpicture}
\end{center}
\caption{$\mathcal{P}^0(S)$}
\end{subfigure}
\begin{subfigure}{0.19\linewidth}
\begin{center}
\begin{tikzpicture}[auto, scale=1]
\tikzstyle{vertex}=[draw, circle, inner sep=0.6mm]

\node (v3) at (108:.8)[vertex, fill=blue, label= $v_3$]{};

\foreach[evaluate=\y using 36*\x)] \x in {2}
\node (v\x) at (\y:.8)[vertex, fill=blue, label=$v_{\x}$]{};

\foreach[evaluate=\y using 36*\x)] \x in {4,5}
\node (v\x) at (\y:.8)[vertex, fill=blue, label=left: $v_{\x}$]{};

\foreach[evaluate=\y using 36*\x)] \x in {6}
\node (v\x) at (\y:.8)[vertex, label=left: $v_{\x}$]{};

\foreach[evaluate=\y using 36*\x)] \x in {7, 8}
\node (v\x) at (\y:.8)[vertex, label=below: $v_{\x}$]{};

\foreach[evaluate=\y using 36*\x)] \x in {1, 9, 10}
\node (v\x) at (\y:.8)[vertex, label=right: $v_{\x}$]{};

\foreach[evaluate=\y using int(\x+1))] \x in {1, ..., 9}\draw (v\x) -- (v\y);

\draw(v10)--(v1);

\foreach \x in {6,7,8}
\draw (v\x) to (v1);
\draw (v2) to (v5);
\end{tikzpicture}
\end{center}
\caption{$\mathcal{P}^1(S)$}
\end{subfigure}
\begin{subfigure}{0.19\linewidth}
\begin{center}
\begin{tikzpicture}[auto, scale=1]
\tikzstyle{vertex}=[draw, circle, inner sep=0.6mm]

\node (v3) at (108:.8)[vertex, fill=blue, label= $v_3$]{};

\foreach[evaluate=\y using 36*\x)] \x in {2}
\node (v\x) at (\y:.8)[vertex, fill=blue, label=$v_{\x}$]{};

\foreach[evaluate=\y using 36*\x)] \x in {4,5, 6}
\node (v\x) at (\y:.8)[vertex, fill=blue, label=left: $v_{\x}$]{};

\foreach[evaluate=\y using 36*\x)] \x in {7, 8}
\node (v\x) at (\y:.8)[vertex, label=below: $v_{\x}$]{};

\foreach[evaluate=\y using 36*\x)] \x in {9, 10}
\node (v\x) at (\y:.8)[vertex, label=right: $v_{\x}$]{};

\foreach[evaluate=\y using 36*\x)] \x in {1}
\node (v\x) at (\y:.8)[vertex, fill=blue, label=right: $v_{\x}$]{};

\foreach[evaluate=\y using int(\x+1))] \x in {1, ..., 9}
\draw (v\x) -- (v\y);

\draw(v10)--(v1);

\foreach \x in {6,7,8}
\draw (v\x) to (v1);
\draw (v2) to (v5);
\end{tikzpicture}\end{center}
\caption{$\mathcal{P}^2(S)$}
\end{subfigure}
\begin{subfigure}{0.19\linewidth}
\begin{center}
\begin{tikzpicture}[auto, scale=1]
\tikzstyle{vertex}=[draw, circle, inner sep=0.6mm]
\node (v3) at (108:.8)[vertex, fill=blue, label= $v_3$]{};

\foreach[evaluate=\y using 36*\x)] \x in {2}
\node (v\x) at (\y:.8)[vertex, fill=blue, label=$v_{\x}$]{};

\foreach[evaluate=\y using 36*\x)] \x in {4,5, 6}
\node (v\x) at (\y:.8)[vertex, fill=blue, label=left: $v_{\x}$]{};

\foreach[evaluate=\y using 36*\x)] \x in {7}
\node (v\x) at (\y:.8)[vertex, fill=blue,  label=below: $v_{\x}$]{};
\foreach[evaluate=\y using 36*\x)] \x in {8}
\node (v\x) at (\y:.8)[vertex, label=below: $v_{\x}$]{};

\foreach[evaluate=\y using 36*\x)] \x in {9, 10}
\node (v\x) at (\y:.8)[vertex, label=right: $v_{\x}$]{};

\foreach[evaluate=\y using 36*\x)] \x in {1}
\node (v\x) at (\y:.8)[vertex, fill=blue, label=right: $v_{\x}$]{};
\foreach[evaluate=\y using int(\x+1))] \x in {1, ..., 9}
\draw (v\x) -- (v\y);

\draw(v10)--(v1);

\foreach \x in {6,7,8}
\draw (v\x) to (v1);
\draw (v2) to (v5);
\end{tikzpicture}\end{center}
\caption{$\mathcal{P}^3(S)$}
\end{subfigure}

\end{center}
\caption{A graph $G$ with $\fgp(G)=0$ as in Theorem \ref{thm:listwith0}, item \ref{fanchordsplus1}.  $S=\{ v_3\}$ is shown in blue on the left, followed by $\mathcal{P}^0(S)$ through $\mathcal{P}^3(S)$.  Continuing, $\mathcal{P}^4(S) = \{v_1, v_2, \ldots v_8\}$, and $\mathcal{P}^5(S)=V$.}
\label{fig:fanchords}
\end{figure}
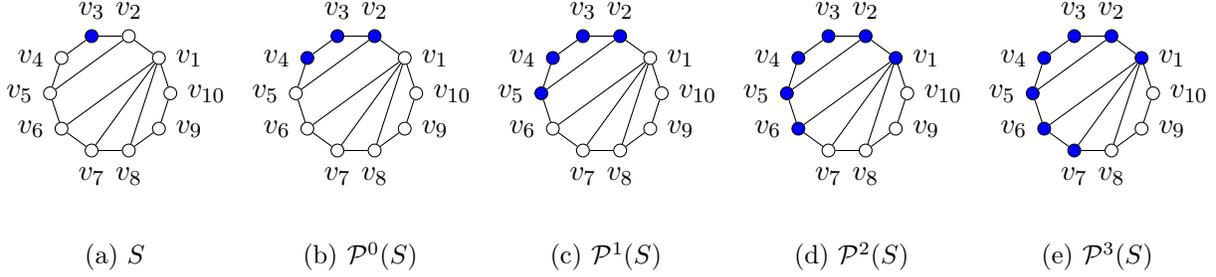

\begin{thm} 
If $G$ is any of the following graphs, then $\fgp(G)=0$.
\begin{enumerate}
\item $\overline{C_n}$ for $n \geq 5$, \label{cyclecomplement}
\item $\overline{P_n}$ for $n \geq 4$, \label{pathcomplement}
\item $C_n = v_1 v_2 \ldots v_n v_1$ with $k$ chords: $\{v_1, v_i\}, \{v_1, v_{i+1}\}, \ldots,\{v_1, v_{i+k-1}\}$, where $i \geq 3$, $n \geq 4$, and $i+k   \leq n-1$.  \label{fanchords}
\item $C_n = v_1 v_2 \ldots v_n v_1$ with $k+1$ chords: $\{v_1, v_i\}, \{v_1, v_{i+1}\}, \ldots, \{v_1, v_{i+k-1}\}$, and $\{ v_2, v_{i-1}\}$ where $i \geq 5$, $n \geq 6$, and $i+k   \leq n-1$.  
   \label{fanchordsplus1}
\item $G \vee H$ where $\fgp(G)=0$ and either $\fgp(H)=0$ or $H= \overline{K_2}$.   \label{fromlemma}
\end{enumerate}
\label{thm:listwith0}
\end{thm}

\begin{proof}

To prove \ref{cyclecomplement}, let $G=\overline{C_n}$ with $n \geq 5$, and $S=\{v\}$ for any vertex $v \in V(G)$.  Let the two neighbors of $v$ in $C_n = \overline{G}$ be  $u$ and $w$.    In $G=\overline{C_n}$, then $\mathcal{P}^0(S) = V(G) \backslash \{u, w\}$, and it follows easily that $\mathcal{P}^1(S)=V(G)$.  

To prove \ref{pathcomplement}, note that if $G=\overline{P_n}$ with $n \geq 4$, and $S = \{ v\}$ for any vertex $v$ with $\deg(v)=2$ in $P_n$, then we can use the same argument as for $\overline{C_n}$.  Otherwise, if $\deg(v)=1$ in $P_n$, then $\mathcal{P}^0(S)=V(G) \backslash \{u\}$ where $u$ is the unique neighbor of $v$ in $P_n$.  Since $n\geq 4$, we see that $\mathcal{P}^1(S)=V$, and $S=\{v\}$ is a PDS.  

To prove \ref{fanchords}, let $P_1$ and $P_2$ be the unique paths from $v_1$ to $v_i$ and from $v_1$ to $v_{i+k-1}$, respectively, whose internal vertices all have degree 2. Note that $\{v_1, v_{j}, v_{j+1}\}$ is a zero forcing set of $G$ for $2 \leq j \leq n-1$. 
If $S= \{v_1\}$, then $\mathcal{P}^0(S)$ consists of $v_1$, $v_2$, $v_n$, and $v_{i}$ through $v_{i+k-1}$, which is a zero forcing set of $G$.  Hence, $S=\{v_1\}$ is a PDS.  Similarly, if $S=\{v_j\}$ where $i \leq j \leq i+k -1$ , then $\mathcal{P}^0(S)$ includes $v_1$, $v_{j-1}$, $v_j$, and $v_{j+1}$, a zero forcing set of $G$.  Hence $S=\{v_j\}$ is a PDS.  
  Finally, suppose $S=\{u\}$ for an internal vertex of $P_1$ or $P_2$.  There exists $j$ such that $v \in \mathcal{P}^j(S)$, and both neighbors of $u$ on the cycle are also in $\mathcal{P}^j(S)$, so $\mathcal{P}^j(S)$ is a zero forcing set, and it follows that $\{u\}$ is a PDS.     

To prove \ref{fanchordsplus1}, note that $\{v_1, v_{j}, v_{j+1}\}$ is a zero forcing set of $G$ for $2 \leq j \leq n-1$.  Let $P_1$ denote the path with all internal vertices of degree 2 from $v_2$ to $v_{i-1}$, and $P_2$ the similar path from $v_{i+k-1}$ to $v_1$.  If $S=\{v_1\}$, then $\mathcal{P}^0(S)$ contains $v_1$ and $\{v_j, i \leq j \leq i+k-1\}$, which is a zero forcing set of $G$; hence $S=\{v_1\}$ is a PDS.  Similarly, if $S=\{v_{\ell}\}$ for $ \ell \in \{ 2, i, i+1, \ldots, i+k-1, n\}$, then $\mathcal{P}^0(S)$ contains $\{v_1, v_j, v_{j+1}\}$ for some $j$, a zero forcing set.  
Thus, $S=\{v_{\ell}\}$ for $\ell \in \{ 2, i, i+1, \ldots, i+k-1, n\}$ is a PDS.  Suppose $S=\{v_{i-1}\}$.  Then $\mathcal{P}^0(S)=\{v_{i-2}, v_{i-1}, v_i, v_2\}$.  Since $v_{i-2}$ has a unique neighbor $v_{i-3}$ outside of $P^0(S)$, the next 
vertex along $P_1$, $\mathcal{P}^1(S)=\{v_{i-3}, v_{i-2}, v_{i-1}, v_i, v_2\}$.  This continues for all internal vertices 
of $P_1$, giving us that for some $\ell$, $\mathcal{P}^{\ell}(S)=V(P_1) \cup \{v_{i}\}$.  Since $v_1$ is the only neighbor of $v_2$ outside of $\mathcal{P}^{\ell}(S)$, $v_1 \in \mathcal{P}^{\ell+1}(S)$, so $\mathcal{P}^{\ell+1}(S)$ includes at least two adjacent vertices in $G$ as well as $v_1$, which is a zero forcing set.  Hence, $S = \{v_{i-1}\}$ is a PDS.  If $S=\{u\}$ where $u \in V(P_1)$ or $u \in V(P_2)$, there exists $\ell$ such that $\mathcal{P}^{\ell}(S)$ contains all vertices in $P_1$ and the vertex $v_1$, or all vertices in $P_2$ (which includes $v_1$).  This is a zero forcing set; hence $S=\{u\}$ is a PDS for any $u \in V(P_1)$ or $V(P_2)$.  For an example, see Figure \ref{fig:fanchords}.  

The proof of \ref{fromlemma} follows immediately from Lemma \ref{lem:join}.
\end{proof}

Note that if $G$ is not connected, then the vertices of any single component form an FPDS of $G$, giving us the following observation.
\begin{obs}
If $\fgp(G)=0$, then $G$ is connected.
\end{obs}

The \emph{path cover number} of a graph $G$, denoted $\mbox{P}(G)$, is the minimum number of vertex disjoint paths, each of which is an induced subgraph of $G$, that contain all vertices of $G$.  Hogben \cite[Theorem 2.13]{hogben2010minimum} showed that $\mbox{P}(G) \leq \Z(G)$, which easily leads to the next theorem.  

\begin{thm}
Suppose $\fgp(G)=0$ and either $G$ has a vertex of degree one or a cut-vertex.  Then $G=P_n$ for some $n \geq 1$. 
\end{thm}

\begin{proof}
Suppose $G$ has a vertex $v$ of degree one, and $\fgp(G)=0$. Then $\mathcal{P}^0(\{v\})=\{u, v\}$ where $u$ is the unique neighbor of $v$.  By Observation \ref{obs:domset}, $\{u\}$ is a zero forcing set of $G[V\backslash \{v \}]$.  Thus $\Z(G[V \backslash \{v \}])=1$, and by
 \cite[Theorem 2.13]{hogben2010minimum}, since $\mbox{P}(G) \leq \Z(G)$, 
$\mbox{P}(G[V \backslash \{v \}])=1$.  Hence, $G[V \backslash \{v \}]$ is a path, and consequently, $G$ is as well.  

Suppose $G$ has a cut-vertex $v$.  Let $u \in V$ with $u \neq v$, and let $\mathcal{K}_u$ be the component of $G[V \backslash \{v \}]$ containing $u$.  Let vertex $w$ be in a different component, $\mathcal{K}_w$ of $G[V \backslash \{v \}]$.  By assumption, both $\{u \}$ and $\{ w\}$ are PDS.  
Then there exists some $j$ such that $v \in \mathcal{P}^j(\{u\})$, and all other vertices in   $\mathcal{P}^j(\{u\})$ are in the component $\mathcal{K}_u$.  Then $|N(v) \backslash V(\mathcal{K}_w)| =1$, because we assumed that $|N(v) \backslash V(\mathcal{K}_w)| \geq 1$, and if $|N(v) \backslash V(\mathcal{K}_w)|  \geq 2$, then $\{u\}$ is an FPDS.  Since we can make the same argument using $w$ instead of $u$, we know that $v$ has exactly two neighbors: $u' \in V(\mathcal{K}_u)$ and $w' \in V(\mathcal{K}_w)$.    The set $S'=\{v\}$ is a PDS by assumption.  Then $\mathcal{P}^0(S') = \{u', v, w'\}$.  $G[V\backslash \{v\}]$ consists of two components, $\mathcal{K}_u$ and $\mathcal{K}_w$, so $\{u'\}$ is a zero forcing set of $\mathcal{K}_u$ and $\{w'\}$ is a zero forcing set of $\mathcal{K}_w$.  Thus, $\mathcal{K}_u$ is a path with end vertex $u'$, and $\mathcal{K}_w$ is a path with end vertex $w'$.  It follows that $G=P_n$.\end{proof}

\begin{cor}
If $\fgp(G)=0$ and $G$ is not a path, then $\kappa(G) \geq 2$.  In particular, $G$ is not a tree.  
\end{cor}

%also note to jonathan's corona of two graphs
\section{Values of $\fgp(G)$ for special graphs}
\label{sec:specificgraphs}
In this section, we determine the value of $\fgp(G)$ for some specific graph families.

\begin{thm}
 The failed power domination number of the complete bipartite graph $K_{m,n}$ with $m \geq n \geq 1$ is given by
 $$\fgp(K_{m,n})=  
 \begin{cases} 
     m-2 & \quad  \mbox{ if } m \geq 2 \\
      0 & \quad \mbox{ otherwise.} \\
   \end{cases}$$

 \end{thm}
\begin{proof}
If $m = 1$,  $G=K_2$, clearly resulting in $\fgp(G)=0$.  If $m=n=2$, $K_{m,n}=C_4$, so $\fgp(K_{m,n} )=0$ by  Theorem \ref{thm:tostado}, item 2.  If $n=1$ but $m\geq 2$, then $\fgp(K_{m,n})=m-2$ by Theorem \ref{thm:highvalues}, item 3.

Assume $m \geq n \geq 2$ and let $S \subseteq V_1$ with $|S|=m-2$.  Let $u,v$ be the vertices in $V_1 \backslash S$.  Then $\mathcal{P}^0(S)=S \cup V_2$, and $V \backslash \mathcal{P}^0(S) = \{u, v\}$.  Since $N(u)=N(v)=V_2$, $\mathcal{P^1}(S)= \mathcal{P}^0(S)$, $S$ is stalled, and $\fgp(G) \geq m-2$.  Note that if $S$ consists of vertices in both $V_1$ and $V_2$,  then $\mathcal{P}^0(S)=V$, so if $S$ is an FPDS, $S \subseteq V_1$ or $S \subseteq V_2$.  Then $|S| \leq m-2$ because if $|S| > m-2$, then $\mathcal{P}^1(S)=V$.  Hence $\fgp(G) = m-2$.\end{proof}

\begin{figure}[h!]
\begin{center}
\begin{minipage}{3.2in}
\begin{center}
\begin{tikzpicture}[auto, scale=0.9]
\tikzstyle{vertex}=[draw, circle, inner sep=0.6mm]
\node (v1) at (0,0) [vertex, label=$u_0v_0$] {};
\node (v2) at (0,-1)  [vertex, label=below: $u_0v_1$] {};
\node (v7) at (1,0)  [vertex, label=$u_1v_0$] {};
\node (v8) at (1,-1)  [vertex, label=below: $u_1v_1$] {};
\node (v13) at (2,0)  [vertex, fill=blue, label=$u_2v_0$] {};
\node (v14) at (2,-1) [vertex, label=below: $u_2v_1$] {};
\node (v19) at (3,0) [vertex, label=$u_3v_0$] {};
\node (v20) at (3,-1) [vertex, label=below: $u_3v_1$] {};
\node (v25) at (4,0)  [vertex, label=$u_4v_0$] {};
\node (v26) at (4,-1)  [vertex, label=below: $u_4v_1$] {};
\node (v31) at (5,0)  [vertex, label=$u_5v_0$] {};
\node (v32) at (5,-1) [vertex, fill=blue, label=below: $u_5v_1$] {};
\node (v37) at (6,0) [vertex, label=$u_6v_0$] {};
\node (v38) at (6,-1)  [vertex, label=below: $u_6v_1$] {};
\node (v43) at (7,0)  [vertex, label=$u_7v_0$] {};
\node (v44) at (7,-1)  [vertex, label=below: $u_7v_1$] {};
\node (v49) at (8,0)  [vertex, label=$u_8v_0$] {};
\node (v50) at (8,-1) [vertex, label=below: $u_8v_1$] {};
\foreach[evaluate=\y using int(\x-1)] \x in {2, 8, 14, 20, 26, 32, 38, 44, 50}
\draw (v\y) to (v\x);
\foreach[evaluate=\z using int(\x-6)] \x in {8, 14, 20, 26, 32, 38, 44, 50}
\draw (v\z) to (v\x);
\foreach[evaluate=\w using int(\y-6)] \y in {7, 13, 19, 25, 31, 37, 43, 49}
\draw (v\w) to (v\y);
\end{tikzpicture}
\end{center}
\end{minipage}
\begin{minipage}{3.2in}
\begin{center}
\begin{tikzpicture}[auto, scale=0.9]
\tikzstyle{vertex}=[draw, circle, inner sep=0.6mm]
\node (v1) at (0,0) [vertex, label=$u_0v_0$] {};
\node (v2) at (0,-1)  [vertex, label=below: $u_0v_1$] {};
\node (v7) at (1,0)  [vertex, label=$u_1v_0$, fill=blue] {};
\node (v8) at (1,-1)  [vertex, label=below: $u_1v_1$] {};
\node (v13) at (2,0)  [vertex, fill=blue, label=$u_2v_0$] {};
\node (v14) at (2,-1) [vertex, label=below: $u_2v_1$, fill=blue] {};
\node (v19) at (3,0) [vertex, label=$u_3v_0$, fill=blue] {};
\node (v20) at (3,-1) [vertex, label=below: $u_3v_1$] {};
\node (v25) at (4,0)  [vertex, label=$u_4v_0$] {};
\node (v26) at (4,-1)  [vertex, label=below: $u_4v_1$, fill=blue] {};
\node (v31) at (5,0)  [vertex, label=$u_5v_0$, fill=blue] {};
\node (v32) at (5,-1) [vertex, fill=blue, label=below: $u_5v_1$] {};
\node (v37) at (6,0) [vertex, label=$u_6v_0$] {};
\node (v38) at (6,-1)  [vertex, label=below: $u_6v_1$, fill=blue] {};
\node (v43) at (7,0)  [vertex, label=$u_7v_0$] {};
\node (v44) at (7,-1)  [vertex, label=below: $u_7v_1$] {};
\node (v49) at (8,0)  [vertex, label=$u_8v_0$] {};
\node (v50) at (8,-1) [vertex, label=below: $u_8v_1$] {};
\foreach[evaluate=\y using int(\x-1)] \x in {2, 8, 14, 20, 26, 32, 38, 44, 50}
\draw (v\y) to (v\x);
\foreach[evaluate=\z using int(\x-6)] \x in {8, 14, 20, 26, 32, 38, 44, 50}
\draw (v\z) to (v\x);
\foreach[evaluate=\w using int(\y-6)] \y in {7, 13, 19, 25, 31, 37, 43, 49}
\draw (v\w) to (v\y);
\end{tikzpicture}
\end{center}
\end{minipage}
\caption{A ladder graph, $P_{9} \square P_2$ with FPDS $S$ in blue on the left and $\mathcal{P}^0(S)$ in blue on the right.}
\label{fig:ladder}
\end{center}
\end{figure}
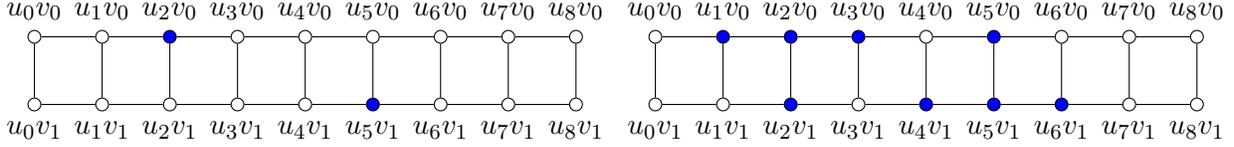

For graphs $G$ and $H$, we denote by $G \square H$ the \emph{Cartesian product} of $G$ and $H$, where $V(G \square H) = V(G) \times V(H)$.  If $u_1, u_2 \in V(G)$ and $v_1, v_2 \in V(H)$, then $(u_1, v_1)$ is adjacent to $(u_2, v_2)$ in $G\square H$ if and only if $u_1=u_2$  in $G$ and $\{v_1,v_2\} \in E(H)$, or $\{u_1, u_2\} \in E(G)$ and $v_1=v_2$ in $H$.  Here, we write $u_iv_j$ for vertex $(u_i, v_j)$ for brevity.  

A \emph{ladder graph} is the graph $P_n \square P_2$ for $n\geq 2$.  Each copy of $P_2$ is called a \emph{rung}.

\begin{thm}
For the ladder graph $G=P_k \square P_2$ with $k\geq 4$, $\fgp(G)= \left\lceil \frac{k-4}{3} \right\rceil$.
\end{thm}
\begin{proof}
Let the vertices of $P_k$ be denoted by $u_i$, $0 \leq i \leq k-1$, and the vertices of $P_2$ by $v_j$, $j=0$ or $1$.  Define $S \subseteq V(P_k \square P_2)$ by $u_iv_j \in S$ if and only if $i \equiv 2 \mod 3$ with $2 \leq i \leq k-3$ and $j \equiv i \mod 2$, as in Figure \ref{fig:ladder}.  Then $|S| =\left\lceil \frac{k-4}{3} \right\rceil$.  We show that $S$ is an FPDS.  Suppose that $u_iv_j \in \mathcal{P}^0(S)$.  If $u_iv_j \in S$, then $N(u_iv_j ) \subseteq \mathcal{P}^0(S)$.  Otherwise, if $u_iv_j \in \mathcal{P}^0(S) \backslash S$, then $u_iv_j$ has exactly two neighbors in $V \backslash \mathcal{P}^0(S)$, namely $u_{i-1}v_j$ or $u_{i+1}v_j$, and $u_i v_{j'}$ where $j' \equiv (j+1) \mod 2$.  Thus $\mathcal{P}^{\infty}(S) = \mathcal{P}^0(S)$; that is, $S$ is an SPDS, giving us $\fgp(G) \geq \lceil \frac{k-4}{3} \rceil$.

To show that $\fgp(G) \leq\lceil \frac{k-4}{3} \rceil$, note that if $\{ u_0 v_0 , u_0 v_1 \} \subseteq Z$, then $Z$ is a zero forcing set, and similarly for $\{u_{k-1} v_0, u_{k-1} v_1\}$.    If $u_0 v_j \in S$, note that $\{ u_0 v_0 , u_0 v_1 \}  \subseteq \mathcal{P}^0(S)$, which implies that $\mathcal{P}^0(S)$ is a zero forcing set, and $S$ is a PDS, and similarly for the case that $u_{k-1} v_j \in S$.  Further, if $u_1 v_j \in S$, then $\{ u_0 v_j , u_1 v_j \}  \subseteq \mathcal{P}^0(S)$, and $\{ u_0 v_0, u_0 v_1, u_1 v_j \}  \subseteq \mathcal{P}^1(S)$, meaning that $\mathcal{P}^1(S)$ is a zero forcing set, and $S$ is a PDS, and similarly for $u_{k-2} v_j \in S$.  Thus, if $S$ is an FPDS, $u_i v_j \notin S$ for $i \in \{0, 1, k-2, k-1\}$ and $j  \in \{0, 1\}$.   That is, no vertices from the first two or last two rungs of the ladder are in any FPDS.  

Also, if $\{ u_i v_0 , u_i v_1,  u_{i+1} v_0 , u_{i+1} v_1  \} \subseteq Z$ for any $i \leq k-2$ (that is, if all vertices from two consecutive rungs are in $Z$), then $Z$ is a zero forcing set.   Thus, if $S$ is an FPDS with $u_i v_j \in S$, then $u_{i-1} v_j, u_{i-1} v_{j'}, u_{i+1} v_{j}, u_{i+1} v_{j'} \notin S$, and further, $u_{i-2} v_{j'}, u_{i+2} v_{j'} \notin S$ for $j' \equiv j \mod 2$.  Suppose that $u_i v_j, u_{i+2} v_j \in S$ for $2 \leq i \leq k-5$ and $j=1$ or $2$.  Then $\{ u_i v_j, u_i v_{j'},  u_{i+1} v_j,  u_{i+2} v_j, u_{i+2} v_{j'} \  \} \subseteq \mathcal{P}^0(S)$ where $j' = (j+1) \mod 2$.  Since $u_{i+1} v_{j'}$ is the only neighbor of $u_{i+1} v_j$ outside of $\mathcal{P}^0(S)$, $u_{i+1} v_j \in \mathcal{P}^1(S)$, giving us that $\{ u_i v_j, u_i v_{j'},  u_{i+1} v_j,  u_{i+1} v_{j'}, u_{i+2} v_j, u_{i+2} v_{j'} \  \} \subseteq \mathcal{P}^1(S)$.  This forms a zero forcing set of $G$; hence, $S$ is a PDS.  Thus, if $S$ is an FPDS with $u_i v_j, u_{i'} v_{j'} \in S$, then $| i - i' | \geq 3$.  That is, $\fgp(G) \leq \lceil \frac{k-4}{3} \rceil$.  
\end{proof}

\begin{thm}
For the graph $G=K_k \square P_\ell$ with $k, \ell \geq 3$, $\fgp(G)=  (k-2) \left\lfloor \frac{\ell-1}{2} \right\rfloor$.
\end{thm}
\begin{proof}
Let the vertices of $K_k$ be denoted $w_i$, $0 \leq i \leq k-1$, and the vertices of $P_\ell$ denoted $x_i$, $0 \leq i \leq \ell-1$.  Define $S \subseteq V(K_k \square P_\ell)$ by $w_i x_j \in S$ if and only if $i \leq k-3$ and $j$ is odd with $j < \ell -1$.  Then $\mathcal{P}^0(S)= V(G \square H) \backslash \left( \{ w_i x_j | i \geq k-2 \mbox{ and } j \mbox{ is even}\} \cup  \{ w_i x_{l-1} |  \mbox{ for any  } i \mbox{ if } l  \mbox{ is even}\}  \right)$.  If $w_i x_j \in S$ with $j$ even, then $w_i x_j$ is adjacent to $w_{k-1} x_j$ and $w_k x_j$; if $w_i x_j \in S$ with $j$ odd and $i < k-2$, then $N(w_i x_j) \subseteq S$; if $w_i x_j \in S$ with $j$ odd and $i \geq k-2$, then $w_i, x_j$ is adjacent to vertices $w_i x_{j-1}$ and $w_i x_{j+1}$, both of which are not in $\mathcal{P}^0(S)$.  Hence $\mathcal{P}^1(S)=\mathcal{P}^0(S)$, and $S$ is an FPDS with $|S|= (k-2) \left\lfloor \frac{\ell-1}{2} \right\rfloor $, giving us $\fgp(G)\geq (k-2) \left\lfloor \frac{\ell-1}{2} \right\rfloor$.  

Now, any set $S'$ with a vertex $w_i x_0$ or $w_i x_{\ell} \in S'$ for any $i$ is a PDS.  Further, if $w_i x_j \in S'$ and $w_i x_{j+1}  \in S'$ for any $j$, then $S'$ is a PDS.  Hence if $S'$ is an FPDS with $|S'| >(k-2) \lfloor \frac{\ell-1}{2} \rfloor $, then there exists some $t$ with $1 \leq t \leq \ell - 2$ and $w_i x_{t} \in S'$ for all $i$ except at most one.  If  $w_i x_{t} \in S'$ for all $i$, then $S'$ is a PDS, thus we must have $w_i x_{t} \notin S$ for some $i$, say $i=k-1$.  Then $\mathcal{P}^0(S')$ includes $w_i  x_{t}$ for all $i$ as well as $w_i x_{t+1}$ and $w_i x_{t-1}$  for all $i<k-1$.  If $w_{k-1} x_{t-1}$ or $w_{k-1} x_{t+1} \in \mathcal{P}^m(S')$ for any $m$, then $S'$ is a PDS.  Hence $2 \leq t \leq \ell-3$, and for every $i$, none of the following vertices are in $S'$: $w_i x_{t-2}$,   $w_i x_{t-1}$,  $w_i x_{t+1}$,  or $w_i x_{t+2}$.  This gives us $|S'| \leq  (k-1) \lceil \frac{\ell-4}{3} \rceil$.  Note that $l \geq 5$ since $l < 5$ implies $S'$ is a PDS.  Since $k \geq 3$, $|S'| <  (k-2) \left\lfloor \frac{\ell-1}{2} \right\rfloor$.  Hence $\fgp(G)=  (k-2) \left\lfloor \frac{\ell-1}{2} \right\rfloor$.   \end{proof}

\section{Future work}

While we were able to produce a list of graphs that have $\fgp(G)=0$ (where every single vertex is itself a PDS), a complete description of all such graphs is still open.  The zero forcing number of trees has been related to other parameters such as the path cover number \cite{aim2008zero}, and 
a technique for determining the zero forcing number of a graph with a cut-vertex was also described \cite{row2012technique}.  Achieving similar results for the failed power domination number of a graph is a feasible problem.  Many parameters in zero forcing, especially related to minimum rank, are investigated for their adherence to a property known as the \emph{Graph Complement Conjecture} which states that the sum of the parameter on $G$ and on the complement graph $\overline{G}$ is bounded by $|V(G)|$ plus a small constant.  For  minimum rank, $\mr(G)$, the conjecture is: $\mr(G)+\mr(\overline{G}) = |V(G)|+2$.  Originally mentioned at an  American Institute for Mathematics workshop \cite{aimworkshop}, it formally appeared in  \cite{barioli2012graph}.   It is natural, and likely challenging, to investigate whether there is any such relationship among power dominating numbers or failed power dominating numbers of graphs and their complements.  

\bibliography{spdibib}
\bibliographystyle{plain}

\end{document}